\documentclass[preprint,12pt]{elsarticle}
\usepackage{fullpage}
\usepackage{hyperref}
\usepackage{amssymb}
\usepackage{amsthm,bm,mathtools}
\usepackage{lineno}
\usepackage{mathrsfs}
\usepackage{amsmath}
\usepackage{color}
\usepackage{arydshln}
\usepackage{caption}
\usepackage{tikz}
\usetikzlibrary{matrix, calc, arrows.meta}

\newtheorem{thm}{Theorem}[section]
\newtheorem{lem}[thm]{Lemma}
\newtheorem{coro}[thm]{Corollary}
\newtheorem{prop}[thm]{Proposition}

\newtheorem{conj}[thm]{Conjecture}

\theoremstyle{definition}

\newtheorem{exm}[thm]{Example}
\newtheorem{rem}[thm]{Remark}
\numberwithin{equation}{section}

\makeatletter

\newcommand{\Rmnum}[1]{\expandafter\@slowromancap\romannumeral #1@}
\newcommand{\mat}[1]{\begin{bmatrix}#1\end{bmatrix}}
\makeatother
\allowdisplaybreaks[4]

\begin{document}

\begin{frontmatter}

\title{Total positivity of transformation matrices for uniform subdivisions}
%\tnotetext[mytitlenote]{xx} %foundation

\author[]{Yanxin Liu}
\ead{liu-yanxin@outlook.com}

\author[]{Jianxi Mao\corref{cor1}}
\ead{maojx@dlut.edu.cn}
\cortext[cor1]{Corresponding author}

\address{School of Mathematical Sciences, Dalian University of Technology, Dalian 116024,\\ P. R. China}

\begin{abstract}
The transformation of the $h$-vector of a finite simplicial complex under an $\mathcal{F}$-uniform subdivision is encoded by a transformation matrix.
Mu and Welker conjectured that the transformation matrix of the barycentric subdivision is totally positive.
In this paper, we give a new combinatorial proof of this conjecture.
We also prove the total positivity of the transformation matrix of the interval subdivision.
In addition, we establish a sufficient condition for the transformation matrix of a uniform subdivision to be totally positive of order $2$ (TP$_2$), thereby partially answering a question of Mu and Welker.
As an application, we show that the transformation matrix of the $r$-colored barycentric subdivision is TP$_2$.
\end{abstract}


\end{frontmatter}
%\linenumbers
%\pagewiselinenumbers
%\switchlinenumbers

%--------------------------------------------------
\section{Introduction}
%--------------------------------------------------
The $h$-vector is one of the central enumerative invariants of a finite simplicial complex.
Since Stanley's work on the enumerative theory of simplicial subdivisions~\cite{Sta92}, a fundamental problem has been to understand how the $h$-vector changes under subdivision operations. 
This viewpoint has motivated the study of $h$-vector transformations for several important families of subdivisions, including the edgewise subdivision~\cite{Ath14, BW09}, the barycentric subdivision~\cite{BW06, JN09, NPT11}, the interval subdivision~\cite{AN20, AN21} and the antiprism triangulation \cite{ABJ22, IJ03}. 
Athanasiadis~\cite{Ath22} introduced the concept of $\mathcal{F}$-uniform subdivision, which provides a unified framework for these classical operations. 
Let 
$$
h^\Delta = (h_0^\Delta, \dots, h_d^\Delta)
$$
be the $h$-vector of the $(d-1)$-dimensional finite simplicial complex $\Delta$.
Athanasiadis \cite{Ath22} showed that 
if a simplicial complex $\Delta'$ is an $\mathcal{F}$-uniform subdivision of $\Delta$, then the $h$-vector of $\Delta'$ can be expressed via a transformation matrix $H_{\mathcal{F}}$,
$$
h^{\Delta'} = H_{\mathcal{F}} h^\Delta.
$$

Mu and Welker~\cite{MW26} proved that the total positivity properties of $H_{\mathcal{F}}$ are closely related to the preservation of specific coefficient inequalities under $\mathcal{F}$-uniform subdivisions.
Following Pinkus \cite{Pin10}, 
a finite or infinite real matrix is called totally positive of order $r$ (or TP$_r$ for short), 
if its minors of all orders $\le r$ are non-negative. 
It is called totally positive (or TP for short) if its minors of all orders are non-negative. 
The theory of total positivity provides a powerful mathematical framework with fundamental applications across diverse fields such as statistics, mechanics, and representation theory~\cite{Kar68}. 
%Within this framework, the study of totally positive matrices has emerged as a central topic~\cite{Pin10}. 

There has been much interest in the total positivity of the transformation matrix $H_{\mathcal{F}}$.
For example, 
the transformation matrix $H_{\mathcal{F}}$ for the $r$-th edgewise subdivision is TP$_2$, as implied by the earlier work of Diaconis and Fulman \cite{DF09}. 
Furthermore, the second author and Wang~\cite{MW22} obtained the stronger result that $H_{\mathcal{F}}$ for the $r$-th edgewise subdivision is TP.
Mu and Welker~\cite{MW26} showed that when $\mathcal{F}$ is the barycentric subdivision, 
its transformation matrix $H_{\mathcal{F}}$ is $\mathrm{TP}_2$.
Then they raised the following conjecture.
\begin{conj}[{\cite[Conjecture 2.7]{MW26}}]\label{CON}
Let $\mathcal{F}$ be the barycentric subdivision.
Then the transformation matrix $H_{\mathcal{F}}$ is TP.
\end{conj}

Very recently, Mu, Wang and Zhu~\cite{MWZ26} settled this conjecture by verifying  sufficient conditions for total positivity. 
In this paper, we give an independent combinatorial proof of Conjecture~\ref{CON}. 
Moreover,
we prove the total positivity of the transformation matrix of the interval subdivision. 
We also establish a sufficient condition for the TP$_2$ property of transformation matrices of uniform subdivisions by the interlacing property.
As an application, we show that the transformation matrix of the $r$-colored barycentric subdivision is TP$_2$.

The rest of the paper is organized as follows. 
Section~\ref{sec2} collects 
the necessary background on Neville elimination, planar networks, and totally 
positive matrices.
Section~\ref{sec3} is devoted to the barycentric 
subdivision. 
More precisely, we obtain a bidiagonal decomposition of $H_{\mathcal{F}}$ for the barycentric subdivision and use planar networks to give a combinatorial proof of Conjecture~\ref{CON}. 
In Section~\ref{sec4}, 
we prove the total positivity of the transformation matrix of the interval subdivision.
We decompose the $H_{\mathcal{F}}$ for the interval subdivision into three TP matrices, one of which is the $H_{\mathcal{F}}$ for the barycentric subdivision.
Then the result follows by Cauchy-Binet formula.
In Section~\ref{sec5}, 
we establish a general criterion for the TP$_2$ property of transformation matrices of uniform subdivisions,  and apply it to the $r$-colored barycentric subdivision.

%--------------------------------------------------
\section{Preliminaries}
\label{sec2}
%--------------------------------------------------
Following Athanasiadis \cite{Ath22}, 
a triangulation $\Delta'$ of a simplicial complex $\Delta$ is said to be \emph{$\mathcal{F}$-uniform} if, when restricted to any two faces of $\Delta$ having the same dimension, 
the resulting triangulations have the same number of faces in every dimension. 
Equivalently, the combinatorial enumeration of the restriction depends only on the dimension of the original face, and not on the particular face chosen.
The $\mathcal{F}$-uniform subdivisions include the edgewise subdivision, 
the barycentric subdivision, 
the interval subdivision and the antiprism triangulation.
We write $\Delta_{\mathcal{F}}$ for $\Delta'$ in this case and speak of $\mathcal{F}$ as a uniform triangulation. 
The following is due to Athanasiadis~\cite{Ath22}.

\begin{lem}[{\cite[Theorem 1.1]{Ath22}}]\label{UniTrianMatrix}
Let $\Delta$ be an $(d-1)$-dimensional simplicial complex,
and $\mathcal{F}$ be a uniform triangulation of $\Delta$. 
Then there is a transformation matrix $H_{\mathcal{F}}$ such that 
\begin{equation*}
h^{\Delta_\mathcal{F}} = H_{\mathcal{F}} \, h^{\Delta}.
\end{equation*}
\end{lem}

Throughout this paper, matrix row and column indices are counted starting from 0.
Neville elimination (NE, for short) is a Gaussian-type elimination procedure that eliminates entries below the main diagonal column by column, using elementary operations between adjacent rows. 
This method is particularly effective when working with certain matrix classes, including TP matrices (see \cite{GT04, KMRR24, MPR22}). 
We refer the reader to \cite{GP92, GP96} for further details on Neville elimination.

For a nonsingular matrix $A=[a(i,j)]_{0\le i,j\le n}$, the NE algorithm consists of $n$ steps and produces the following sequence of matrices,
\begin{equation} 
\label{seqNE}
A =: A^{(0)} \rightarrow A^{(1)} \rightarrow \cdots \rightarrow A^{(n)} = U,
\end{equation}
where $U$ is an upper triangular matrix and $A^{(k)}$ has zeros below the main diagonal in the first $k$ columns for $1 \leq k \leq n$. 
The matrix $$A^{(k+1)}=\left[a^{(k+1)}(i,j)\right]_{0\le i,j\le n}$$ is obtained from the matrix $A^{(k)}$ by using the formula
\begin{equation*}
a^{(k+1)} (i,j)= \left\{ 
\begin{array}{lcc}  
a^{(k)}(i,j)-\frac{a^{(k)}(i,k)}{a^{(k)}(i-1,k)}\,a^{(k)}(i-1,j),
\, & \textrm{if }\, \begin{aligned}
&k+1\le i\le n,
k\le j\le n,\\
&\textrm{and }\, a^{(k)}(i-1,k)\neq 0,
\end{aligned} \\ 
a^{(k)}(i,j), & \text{otherwise}.
\end{array} \right.
\end{equation*} 
The $(i,j)$ multiplier of NE of $A$, with $0\leq j\leq i \leq n$, is defined as
\begin{equation}
\label{nemulti}
m_{ij} = \left\{ \begin{array}{lcc}  
\frac{a^{(j)}(i,j)}{a^{(j)}(i-1,j)},
\quad& \textrm{if}& a^{(j)}(i-1,j) \not= 0,\\ 
0,\quad & \textrm{if}& a^{(j)}(i-1,j) = 0. 
\end{array} \right.
\end{equation}
The multipliers satisfy the following property: if $m_{i,j}=0$,
then $m_{h,j}=0$ for $h>i$. 
Using NE, Gasca and Pe{\~n}a \cite{GP96} obtained the following matrix factorization for a nonsingular matrix
\begin{equation*}
A = F_{n-1} F_{n-2} \cdots F_{0} U, 
\end{equation*}
where 
\begin{equation*}
F_{i} = \begin{bmatrix}
	1 &        &        &           &        &           & \\
	0 & 1      &        &           &        &           & \\
	& \ddots & \ddots &           &        &           & \\ 
	&        & 0      & 1         &        &           & \\
	&        &        & m_{i+1,0} & 1      &           & \\
	&        &        &           & \ddots & \ddots    & \\
	&        &        &           &        & m_{n, n-i-1} & 1
\end{bmatrix}
\end{equation*}
for $i=0,1,\ldots,n-1$. 

\begin{exm}\label{deco:A4}
Taking 
$$
A=\mat{
	1&11&11&1&0\\
	0&8&14&2&0\\
	0&4&16&4&0\\
	0&2&14&8&0\\
	0&1&11&11&1
}
$$ as an example.
Using Neville elimination,
\begin{align*}
&A =: A^{(0)}=\mat{
	1&11&11&1&0\\
	0&8&14&2&0\\
	0&4&16&4&0\\
	0&2&14&8&0\\
	0&1&11&11&1
}
\xrightarrow[m_{i,0}=0,\ 1\le i\le 4]
{r_i-0\cdot r_{i-1},\ 1\le i\le 4}
A^{(1)}=
\mat{
	1&11&11&1&0\\
	0&8&14&2&0\\
	0&4&16&4&0\\
	0&2&14&8&0\\
	0&1&11&11&1
} \\
&\xrightarrow[m_{4,1}=\frac{1}{2}]
{r_4-\frac{1}{2} r_3}
\mat{
	1&11&11&1&0\\
	0&8&14&2&0\\
	0&4&16&4&0\\
	0&2&14&8&0\\
	0&0&4&7&1
}
\xrightarrow[m_{3,1}=\frac{1}{2}]
{r_3-\frac{1}{2} r_2}
\mat{
	1&11&11&1&0\\
	0&8&14&2&0\\
	0&4&16&4&0\\
	0&0&6&6&0\\
	0&0&4&7&1
}
\xrightarrow[m_{2,1}=\frac{1}{2}]
{r_2-\frac{1}{2} r_1}
A^{(2)}=
\mat{
	1&11&11&1&0\\
	0&8&14&2&0\\
	0&0&9&3&0\\
	0&0&6&6&0\\
	0&0&4&7&1
}\\
&\xrightarrow[m_{4,2}=\frac{2}{3}]
{r_4-\frac{2}{3} r_3}
\mat{
	1&11&11&1&0\\
	0&8&14&2&0\\
	0&0&9&3&0\\
	0&0&6&6&0\\
	0&0&0&3&1
}
\xrightarrow[m_{3,2}=\frac{2}{3}]
{r_3-\frac{2}{3} r_2}
A^{(3)}=
\mat{
	1&11&11&1&0\\
	0&8&14&2&0\\
	0&0&9&3&0\\
	0&0&0&4&0\\
	0&0&0&3&1
}\\
&\xrightarrow[m_{4,3}=\frac{3}{4}]
{r_4-\frac{3}{4} r_3}
A^{(4)}=U=\mat{
	1&11&11&1&0\\
	0&8&14&2&0\\
	0&0&9&3&0\\
	0&0&0&4&0\\
	0&0&0&0&1
}.
\end{align*}
Thus, the factorization of the matrix $A$ equals
\begin{align*}
\mat{
	1& & & & \\
	&1& & & \\
	& &1& & \\
	& & &1& \\
	& & &0&1
}
\mat{
	1& & & & \\
	&1& & & \\
	& &1& & \\
	& & 0 &1& \\
	& & &\frac{1}{2}&1
}
\mat{
	1& & & & \\
	&1& & & \\
	& 0&1& & \\
	& &\frac{1}{2}&1& \\
	& & &\frac{2}{3}&1
}
\mat{
	1& & & & \\
	0&1& & & \\
	&\frac{1}{2}&1&&\\
	&&\frac{2}{3}&1&\\
	&&&\frac{3}{4}&1
}
\mat{
	1&11&11&1&0\\
	0&8&14&2&0\\
	0&0&9&3&0\\
	0&0&0&4&0\\
	0&0&0&0&1
}.
\end{align*}
\end{exm}

Planar networks provide a powerful tool for proving the total positivity of a matrix. 
Let
$(\mathcal{V}, \mathcal{A})$ be a \emph{directed graph} (or \emph{digraph}, for short) with vertex set $\mathcal{V}$ and arc set $\mathcal{A}$ 
without
loops or multiple edges.
A \emph{path} in $(\mathcal{V}, \mathcal{A})$ is a sequence
$
\pi=u_0 u_1 \cdots u_k
$
of vertices such that $(u_i,u_{i+1})\in \mathcal{A}$, for $i=0,\dots,k-1$.
In this case, we say that $\pi$ is a path from $u_0$ to $u_k$. 
The digraph is called \emph{locally finite} if, for any two vertices
$u,v\in \mathcal{V}$, the number of paths from $u$ to $v$ is finite.
If the digraph is planar, locally finite and has a non-negative weight function $w: \mathcal{A}\to \mathbb{R}^{\ge 0}$, then we call $\mathcal{D}=(\mathcal{V}, \mathcal{A}, w)$ a \emph{planar network}. 

Let $\mathcal{D}=(\mathcal{V}, \mathcal{A}, w)$ be a planar network.
The \emph{weight of a path} 
is defined to be the product of the weights of all its components. 
For $u,v\in \mathcal{V}$, let $P_D(u,v)$ be the sum of the weights of all paths from $u$ to $v$. 
Given two vertex subsets
\begin{equation*}
\mathbf{u}=(u_0, u_1,\ldots,u_n),
\quad
\mathbf{v}=(v_0, v_1, \ldots,v_n),
\end{equation*}
we say that $\mathbf{u}$ and $\mathbf{v}$ are \emph{compatible} if,
for every non-identity permutation on $\{0,1,\dots,n\}$, there are no $(n+1)$-tuples of paths from
$(u_0, u_1,\dots,u_n)$ to $(v_{\sigma(0)}, v_{\sigma(1)}, \dots,v_{\sigma(n)})$ that are non-intersecting. 
Moreover, \(\mathbf{u}\) and \(\mathbf{v}\) are called \emph{fully compatible} if
every pair of subsequences
\begin{equation*}
(u_{i_0}, u_{i_1},\dots,u_{i_s}),
\quad
(v_{j_0}, v_{j_1},\dots,v_{j_s})
\end{equation*}
is compatible for all $i_0<\cdots<i_s, \ j_0<\cdots<j_s$.

\begin{thm}[{\cite[Theorem 3.1]{Bre95}}]\label{Bre}
Let $A=\left[a(i,j)\right]_{0\le i,j\le n}$ be a real matrix. Then $A$ is TP if and only if there exists a planar network $\mathcal{D}=(\mathcal{V}, \mathcal{A}, w)$ and two vertex sequences $\mathbf{u}=(u_0, u_1,\dots,u_n),\ \mathbf{v}=(v_0, v_1, \dots,v_n)$ that are fully compatible, such that
\begin{equation*}
a(i,j)=P_D(u_i,v_j)\quad \text{for } \, 0\le i,j\le n.
\end{equation*}
\end{thm}

Using the Lindstr{\"o}m--Gessel--Viennot lemma, Brenti constructed the weighted digraph of a matrix when it has a bidiagonal decomposition. 
We refer the reader to the proof of Theorem 3.1 in~\cite{Bre95}. 
Suppose a matrix $A$ has a bidiagonal decomposition
\begin{equation*}
A=W_0 W_1 \cdots W_{t+m},
\end{equation*}
for any $t,m\ge0$,
where
\begin{equation*}
W_i=
\begin{bmatrix}
a_0^{(i)}  &  &  &  & \\
b_1^{(i)}  & a_1^{(i)} &  &  & \\
& \ddots & \ddots &  & \\
&  & b_{n-1}^{(i)} & a_{n-1}^{(i)} & \\
&  &  & b_n^{(i)} & a_n^{(i)}
\end{bmatrix}
\end{equation*} 
for $i=0,1,\dots,t$, and
\begin{equation*}
W_i=
\begin{bmatrix}
a_0^{(i)}  & b_0^{(i)} &  &  & \\
& a_1^{(i)} & b_1^{(i)} &  & \\
&   & \ddots & \ddots & \\
&  &   & a_{n-1}^{(i)} &  b_{n-1}^{(i)} \\
&  & &  & a_n^{(i)}
\end{bmatrix}
\end{equation*}
for $i=t+1,t+2,\dots,t+m$. 
The planar network $\mathcal{D}=(\mathcal{V}, \mathcal{A}, w)$ associated to $A$ is constructed as follows.
\begin{itemize}
\item Vertex set $\mathcal{V}=\{0,1,\dots, t+m+1\} \times \{0,1,\dots,n\}$.
\item For every vertex $(i, j) \in \{0,1,\dots, t+m\} \times \{0,1,\dots,n\}$, we add a directed edge to $(i+1, j)$ with weight $a_j^{(i)}$, 
add a directed edge from $(i,j)$ to $(i+1,j-1)$
(respectively, $(i+ 1,j+1)$) with weight $b_j^{(i)}$ for $(i,j) \in \{0,1,\dots, t\} \times \{0,1,\dots,n\}$ (respectively, $(i,j) \in \{t+1,\dots, t+m\} \times \{0,1,\dots,n\}$).
\end{itemize}
Let $\mathbf{u}=((0,0), (0,1), \dots, (0,n))$ and $\mathbf{v}=((t+m+1, 0), \dots, (t+m+1, n))$. 
It is clear that $\mathcal{D}$ is a weighted digraph and $\mathbf{u}$ and $\mathbf{v}$ are fully compatible. 
Suppose that all the entries in $W_i$ are non-negative for all $i$.
Then by Theorem~\ref{Bre}, the matrix $A$ is TP. 
For the case $n=t=m=2$, see Figure~\ref{fig:1}.

\begin{figure}[htbp] % htbp 表示排版位置建议
\centering       % 关键居中命令
\begin{tikzpicture}[
scale=1.6,
every node/.style={font=\large},
dot/.style={circle, fill=black, inner sep=2pt},
lbl/.style={font=\scriptsize}
]

\tikzset{
edge/.style={
	line width=1.5pt,
	-{Latex[length=2mm,width=2mm]},
	shorten <=1pt,
	shorten >=4pt
}
}

% 顶点坐标
\foreach \x in {0,...,5}{
\foreach \y in {0,1,2}{
	\coordinate (v\x\y) at (\x,\y);
}
}

% y = 2
\draw[edge] (v02) -- (v12) node[lbl, midway, above=-3pt] {$a_2^{(0)}$};
\draw[edge] (v12) -- (v22) node[lbl, midway, above=-3pt] {$a_2^{(1)}$};
\draw[edge] (v22) -- (v32) node[lbl, midway, above=-3pt] {$a_2^{(2)}$};
\draw[edge] (v32) -- (v42) node[lbl, midway, above=-3pt] {$a_2^{(3)}$};
\draw[edge] (v42) -- (v52) node[lbl, midway, above=-3pt] {$a_2^{(4)}$};

% y = 1
\draw[edge] (v01) -- (v11) node[lbl, midway, above=-3pt] {$a_1^{(0)}$};
\draw[edge] (v11) -- (v21) node[lbl, midway, above=-3pt] {$a_1^{(1)}$};
\draw[edge] (v21) -- (v31) node[lbl, midway, above=-3pt] {$a_1^{(2)}$};
\draw[edge] (v31) -- (v41) node[lbl, midway, above=-3pt] {$a_1^{(3)}$};
\draw[edge] (v41) -- (v51) node[lbl, midway, above=-3pt] {$a_1^{(4)}$};

% y = 0
\draw[edge] (v00) -- (v10) node[lbl, midway, above=-3pt] {$a_0^{(0)}$};
\draw[edge] (v10) -- (v20) node[lbl, midway, above=-3pt] {$a_0^{(1)}$};
\draw[edge] (v20) -- (v30) node[lbl, midway, above=-3pt] {$a_0^{(2)}$};
\draw[edge] (v30) -- (v40) node[lbl, midway, above=-3pt] {$a_0^{(3)}$};
\draw[edge] (v40) -- (v50) node[lbl, midway, above=-3pt] {$a_0^{(4)}$};

% --------------------
% 斜边标签
% --------------------

\node[font=\scriptsize, above=0pt] at (0.65,1.4) {$b_2^{(0)}$};
\node[font=\scriptsize, above=0pt] at (1.65,1.4) {$b_2^{(1)}$};
\node[font=\scriptsize, above=0pt] at (2.65,1.4) {$b_2^{(2)}$};
\node[font=\scriptsize, above=0pt] at (3.4,1.4) {$b_1^{(3)}$};
\node[font=\scriptsize, above=0pt] at (4.4,1.4) {$b_1^{(4)}$};

\node[font=\scriptsize, above=0pt] at (0.65,0.4) {$b_1^{(0)}$};
\node[font=\scriptsize, above=0pt] at (1.65,0.4) {$b_1^{(1)}$};
\node[font=\scriptsize, above=0pt] at (2.65,0.4) {$b_1^{(2)}$};
\node[font=\scriptsize, above=0pt] at (3.4,0.4) {$b_0^{(3)}$};
\node[font=\scriptsize, above=0pt] at (4.4,0.4) {$b_0^{(4)}$};

\draw[edge] (v02) -- (v11);
\draw[edge] (v12) -- (v21);
\draw[edge] (v22) -- (v31);
\draw[edge] (v01) -- (v10);
\draw[edge] (v11) -- (v20);
\draw[edge] (v21) -- (v30);
\draw[edge] (v30) -- (v41);
\draw[edge] (v40) -- (v51);
\draw[edge] (v31) -- (v42);
\draw[edge] (v41) -- (v52);

% 顶点
\foreach \x in {0,...,5}{
\foreach \y in {0,1,2}{
	\node[dot] at (v\x\y) {};
}
}

% 坐标标签
\node[font=\small, left=2pt] at (v02) {$(0,2)$};
\node[font=\small, left=2pt] at (v01) {$(0,1)$};
\node[font=\small, left=2pt] at (v00) {$(0,0)$};
\node[font=\small, right=2pt] at (v52) {$(5,2)$};
\node[font=\small, right=2pt] at (v51) {$(5,1)$};
\node[font=\small, right=2pt] at (v50) {$(5,0)$};
\end{tikzpicture}
\captionsetup{labelsep=none}
\caption{} % 图片标题
\label{fig:1}
\end{figure}

%---------------------------------------------------------------------
\section{Total positivity of the transformation matrix of the barycentric subdivision}
\label{sec3}
%---------------------------------------------------------------------
In this section, we present the bidiagonal decomposition of $H_{\mathcal{F}}$ of the barycentric subdivision by Neville elimination.
Then we give a combinatorial proof of the total positivity of $H_{\mathcal{F}}$ by a planar network.
First we present the explicit expression of the entries of $H_{\mathcal{F}}$,
which can be obtained in~\cite[Theorem 1 and Proof of Theorem 2]{BW06}.

\begin{lem}
Let $\mathcal{F}$ be the barycentric subdivision of $(d-1)$-dimensional simplicial complexes. 
Then we have 
\begin{equation}
\label{exbar}
H_\mathcal{F} = \left[ \sum_{r=0}^{i} r^j (r+1)^{d-j} (-1)^{i-r} \binom{d+1}{i-r} \right ]_{0 \le i,j \le d}.
\end{equation}
\end{lem}

\begin{thm}\label{key}
Let $\mathcal{F}$ be the barycentric subdivision of $(d-1)$-dimensional simplicial complexes. 
Then the transformation matrix $H_\mathcal{F}$ has a bidiagonal decomposition
$$
H_\mathcal{F}=W_{0} W_{1} \cdots W_{d-1} W_d \cdots W_{2d-1},
$$
where
\begin{equation*}
W_i =
\begin{tikzpicture}[baseline=(M.center)]
\matrix (M) [
matrix of math nodes,
nodes in empty cells,
left delimiter={[},
right delimiter={]},
column sep=0.7em,
row sep=0.35em,
nodes={
anchor=center,
inner sep=0pt,
minimum width=1.1em,
minimum height=1.05em,
font=\scriptsize
}
] {
1&       &        &   &             &        &      \\
d-i-1 & 2 &   &   &             &        &      \\
&  \ddots  & \ddots   &   &             &        &      \\
&        &    1    & d-i &   &        &      \\
&        &        &  0 & 1      &  &      \\
&        &        &   &    \ddots         & \ddots    &   \\
&        &        &   &             &     0   & 1\\
};

\node[font=\scriptsize] at ($(M-1-4.north)+(0,10pt)$) {$d-i-1$};
\node[font=\scriptsize, anchor=west] at ($(M-4-7.east)+(15pt,0)$) {$d-i-1$};
\end{tikzpicture},
\end{equation*}
for $i=0,1,\dots,d-1$, and
\begin{equation*}
W_{d+j} =
\begin{tikzpicture}[baseline=(M.center)]
\matrix (M) [
matrix of math nodes,
nodes in empty cells,
left delimiter={[},
right delimiter={]},
column sep=0.7em,
row sep=0.35em,
nodes={
anchor=center,
inner sep=0pt,
minimum width=1.1em,
minimum height=1.05em,
font=\scriptsize
}
] {
1&   0    &        &   &             &        &      \\
& \ddots & \ddots &   &             &        &      \\
&        & 1      & 0 &             &        &      \\
&        &        & 1 & \frac{1}{2} &        &      \\
&        &        &   & \ddots      & \ddots &      \\
&        &        &   &             & 1      &  \frac{d-j-1}{d-j}   \\
&        &        &   &             &        & 1\\
};

\node[font=\scriptsize] at ($(M-1-4.north)+(0,10pt)$) {$j+1$};
\node[font=\scriptsize, anchor=west] at ($(M-4-7.east)+(15pt,0)$) {$j+1$};
\end{tikzpicture}
\end{equation*}
for $j=0,1,\dots,d-1$. 
\end{thm}

Using Brenti's construction, we give a planar network associated with the transformation matrix $H_\mathcal{F}$. 
See Figure~\ref{fig:2}. 
By Theorem~\ref{Bre}, we provide an alternative combinatorial proof of Conjecture~\ref{CON}.

\begin{coro}\label{TP-sub}
Let $\mathcal{F}$ be the barycentric subdivision of $(d-1)$-dimensional simplicial complexes.
Then the transformation matrix $H_{\mathcal{F}}$ is totally positive.
\end{coro} 

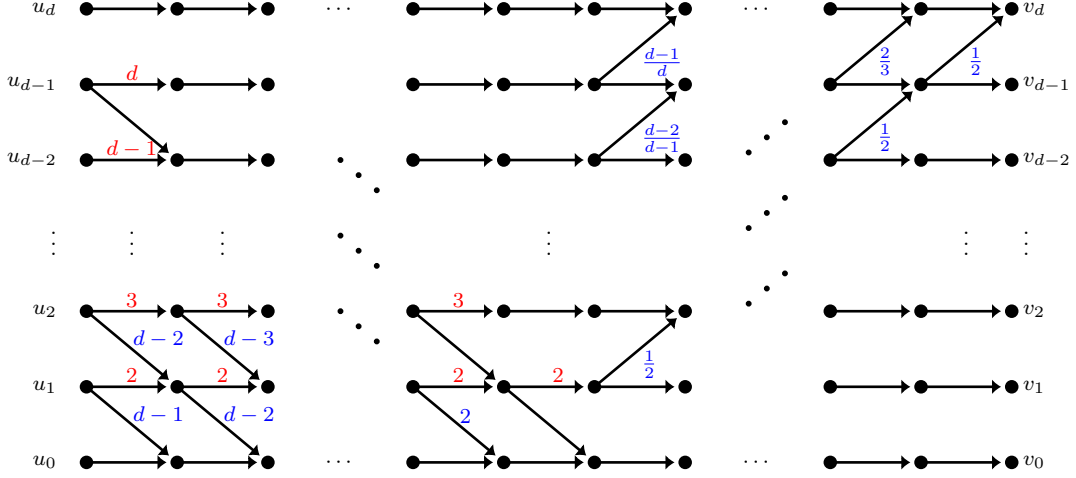
\begin{figure}[htbp] % htbp 表示排版位置建议
\centering       % 关键居中命令
\begin{tikzpicture}[
x=1.2cm, y=1.0cm,
thick,
font=\scriptsize
]

\tikzset{
n/.style={circle, fill=black, inner sep=1.8pt},
edge/.style={
line width=1pt,
-{Latex[length=1.2mm,width=2mm]},
shorten <=1pt,
shorten >=4pt
}
}

% ---------------- 绘制水平实线 ----------------
\foreach \y in {0, 1, 2, 4, 5, 6} {
% 左边区域
\draw[edge] (0,\y) -- (1,\y);
\draw[edge] (1,\y) -- (2,\y);

% 中间区域
\draw[edge] (3.6,\y) -- (4.6,\y);
\draw[edge] (4.6,\y) -- (5.6,\y);
\draw[edge] (5.6,\y) -- (6.6,\y);

% 右边区域
\draw[edge] (8.2,\y) -- (9.2,\y);
\draw[edge] (9.2,\y) -- (10.2,\y);
}

% ---------------- 绘制水平省略号 ----------------
\node at (2.8,0) {$\cdots$};
\node at (2.8,6) {$\cdots$};

\node at (7.4,0) {$\cdots$};
\node at (7.4,6) {$\cdots$};

% ---------------- 绘制所有实心节点 ----------------
\foreach \y in {0, 1, 2, 4, 5, 6} {
% 左侧 3 列点
\foreach \x in {0, 1, 2} {
\node[n] at (\x,\y) {};
}

% 中间 4 列点
\foreach \x in {3.6, 4.6, 5.6, 6.6} {
\node[n] at (\x,\y) {};
}

% 右侧 3 列点
\foreach \x in {8.2, 9.2, 10.2} {
\node[n] at (\x,\y) {};
}
}

% ---------------- 绘制行标签 ----------------
\node[left] at (-0.2,6) {$u_d$};
\node[left] at (-0.2,5) {$u_{d-1}$};
\node[left] at (-0.2,4) {$u_{d-2}$};
\node[left] at (-0.2,3) {$\vdots$};
\node[left] at (-0.2,2) {$u_2$};
\node[left] at (-0.2,1) {$u_1$};
\node[left] at (-0.2,0) {$u_0$};

\node[right] at (10.2,6) {$v_d$};
\node[right] at (10.2,5) {$v_{d-1}$};
\node[right] at (10.2,4) {$v_{d-2}$};
\node[right] at (10.2,3) {$\vdots$};
\node[right] at (10.2,2) {$v_2$};
\node[right] at (10.2,1) {$v_1$};
\node[right] at (10.2,0) {$v_0$};

% ---------------- 绘制区域间的垂直省略号 ----------------
\node at (0.5,3) {$\vdots$};
\node at (1.5,3) {$\vdots$};
\node at (5.1,3) {$\vdots$};
\node at (9.7,3) {$\vdots$};

% ---------------- 绘制对角连线及权重标签 ----------------

% 左侧区域第 1 段
\draw[edge] (0,5) -- (1,4);
\draw[edge] (0,2) -- (1,1) node[midway, above right, inner sep=0pt] {\textcolor{blue}{$d-2$}};
\draw[edge] (0,1) -- (1,0) node[midway, above right, inner sep=0pt] {\textcolor{blue}{$d-1$}};
\draw[edge] (1,1) -- (2,0) node[midway, above right, inner sep=0pt] {\textcolor{blue}{$d-2$}};
\draw[edge] (1,2) -- (2,1) node[midway, above right, inner sep=0pt] {\textcolor{blue}{$d-3$}};

\node at (0.5,1.15) {\textcolor{red}{$2$}};
\node at (1.5,1.15) {\textcolor{red}{$2$}};
\node at (1.5,2.15) {\textcolor{red}{$3$}};
\node at (0.5,2.15) {\textcolor{red}{$3$}};
\node at (0.5,4.15) {\textcolor{red}{$d-1$}};
\node at (0.5,5.15) {\textcolor{red}{$d$}};

% 中间区域第 1 段
\draw[edge] (3.6,1) -- (4.6,0) node[midway, above right, inner sep=0pt] {\textcolor{blue}{$2$}};
\draw[edge] (3.6,2) -- (4.6,1);

\node at (4.1,1.15) {\textcolor{red}{$2$}};
\node at (4.1,2.15) {\textcolor{red}{$3$}};

% 中间区域第 2 段
\draw[edge] (4.6,1) -- (5.6,0);
\node at (5.2,1.15) {\textcolor{red}{$2$}};

% 中间区域第 3 段
\draw[edge] (5.6,5) -- (6.6,6) node[midway, below right, inner sep=0pt] {\textcolor{blue}{$\frac{d-1}{d}$}};
\draw[edge] (5.6,4) -- (6.6,5) node[midway, below right, inner sep=0pt] {\textcolor{blue}{$\frac{d-2}{d-1}$}};
\draw[edge] (5.6,1) -- (6.6,2) node[midway, below right, inner sep=0pt] {\textcolor{blue}{$\frac{1}{2}$}};

% 右侧区域
\draw[edge] (9.2,5) -- (10.2,6) node[midway, below right, inner sep=0pt] {\textcolor{blue}{$\frac{1}{2}$}};
\draw[edge] (8.2,5) -- (9.2,6) node[midway, below right, inner sep=0pt] {\textcolor{blue}{$\frac{2}{3}$}};
\draw[edge] (8.2,4) -- (9.2,5) node[midway, below right, inner sep=0pt] {\textcolor{blue}{$\frac{1}{2}$}};

% ---------------- 绘制斜向对角线延长线上的点 ----------------
\foreach \dy in {0, 1, 2} {
\filldraw (3.2, 1.6 + \dy) circle (0.7pt);
\filldraw (3.0, 1.8 + \dy) circle (0.7pt);
\filldraw (2.8, 2.0 + \dy) circle (0.7pt);
}

\foreach \dy in {0, 1, 2} {
\filldraw (7.7, 4.5 - \dy) circle (0.7pt);
\filldraw (7.5, 4.3 - \dy) circle (0.7pt);
\filldraw (7.3, 4.1 - \dy) circle (0.7pt);
}

\end{tikzpicture}
\caption{The Planar network for $H_{\mathcal{F}}$ of the barycentric subdivision} % 图片标题
\label{fig:2}
\end{figure}

In the following, we give a proof of Theorem~\ref{key}.
Fix a positive integer $d$. 
Let $A_d= \left [a_d(i,j) \right ]_{0 \le i,j \le d}$ be 
the transpose of $H_\mathcal{F}$ in~\eqref{exbar}, with entries 
$$a_d(i,j) = \sum_{r=0}^{j} r^i (r+1)^{d-i} (-1)^{j-r} \binom{d+1}{j-r}.$$
To simplify the proof, we consider the bidiagonal decomposition of $A_d$.

%---------------------------------------------------------------------
\subsection{Neville elimination on the transpose of $H_\mathcal{F}$}
%---------------------------------------------------------------------
To facilitate the proof, let $T_d$ be a linear row operator: for any polynomial $f(x)$ of degree at most $d$, define the row vector 
\begin{equation*}
T_d(f) = (\left ( T_d(f) \right )_0, \left ( T_d(f) \right )_1, \dots, \left ( T_d(f) \right )_d),
\end{equation*}
where the $j$-th component is:
\begin{equation}
\left ( T_d(f) \right )_j = \sum_{r=0}^{j} f(r) (-1)^{j-r} \binom{d+1}{j-r}.
\end{equation}
Using this operator, the $i$-th row of $A_d$ is precisely $T_d(x^i(x+1)^{d-i})$. 

\begin{lem}\label{NEofAd}
For $1\le k \le d$, let 
$$
A_d^{(k)}=\left [ a_d^{(k)}(i,j) \right ]_{0 \le i,j \le d}
$$ 
be the matrix obtained after $k$ steps of Neville elimination of $A_d$.
Then for $k \leq i \leq d, \ k-1 \leq j\leq d$,
we have 
\begin{equation}\label{eNEforAd}
a_d^{(k)}(i,j) = \left(T_d(f^{(k)}_i(x))\right)_j,
\end{equation}
where 
$$
f^{(k)}_i(x)=\binom{x}{k}x^{i-k}(x+1)^{d-i}.
$$
\end{lem}

\begin{proof}
We prove the statement by induction on $k$. 
For $k=1$, since 
\begin{equation*}
a_d(i,0)=\sum_{r=0}^{0}\binom{r}{0}r^{i}(r+1)^{d-i}(-1)^{0-r}\binom{d+1}{0-r}=
\begin{cases}
1 & \text{ if } i=0, \\
0 & \text{otherwise},
\end{cases}
\end{equation*}
we obtain that 
\begin{align*}
a_d^{(1)}(i,j)=a_d(i,j)&=\left(T_d \left( \binom{x}{0}x^{i}(x+1)^{d-i}\right) \right)_j\\
&=\left(T_d \left( \binom{x}{1}x^{i-1}(x+1)^{d-i}\right) \right)_j=\left(T_d(f^{(1)}_i(x))\right)_j
\end{align*}
for $1 \leq i \leq d, 0\leq j\leq d$. 
For $k=2$, consider $i\ge 1$.
Then
\begin{equation*}
a^{(1)}_d(i,1)=a_d(i,1)=\sum_{r=0}^{1}r^{i}(r+1)^{d-i}(-1)^{1-r}\binom{d+1}{1-r}=2^{d-i}.
\end{equation*}
Applying NE, for $2 \leq i \leq d, 1 \leq j\leq d$, we have 
\begin{align*}
a_d^{(2)}(i,j)=& a_d^{(1)}(i,j)-\frac{ a_d^{(1)}(i,1)}{ a_d^{(1)}(i-1,1)} a_d^{(1)}(i-1,j)\\
=&\left(T_d\left(f^{(1)}_i(x)-\frac{1}{2}f^{(1)}_{i-1}(x)\right)\right)_j\\
=&\left(T_d\left(\binom{x}{1}x^{i-1}(x+1)^{d-i}-\frac{1}{2}\binom{x}{1}x^{i-2}(x+1)^{d-i+1}\right)\right)_j\\
=&\left(T_d \left( \binom{x}{2}x^{i-2}(x+1)^{d-i}\right) \right)_j
=\left(T_d(f^{(2)}_i(x))\right)_j
\end{align*}
by the linearity of $T_d$. 

Assume that \eqref{eNEforAd} holds for $k$. 
Performing the $(k+1)$-th step of NE yields
\begin{equation*}
a_d^{(k+1)}(i,j)=a_d^{(k)}(i,j)-\frac{ a_d^{(k)}(i,k)}{ a_d^{(k)}(i-1,k)} a_d^{(k)}(i-1,j).
\end{equation*}
By the induction hypothesis, 
\begin{equation}\label{multi}
a_d^{(k)}(i,k)=\sum_{r=0}^{k}\binom{r}{k}r^{i-k}(r+1)^{d-i}(-1)^{k-r}\binom{d+1}{k-r}
=k^{i-k}(k+1)^{d-i}.
\end{equation}
Then we obtain that
\begin{equation*}
a_d^{(k+1)}(i,j)= a_d^{(k)}(i,j)-\frac{ a_d^{(k)}(i,k)}{ a_d^{(k)}(i-1,k)} a_d^{(k)}(i-1,j)
=\left(T_d\left[f^{(k)}_i(x)-\frac{k}{k+1}f^{(k)}_{i-1}(x)\right]\right)_j.
\end{equation*}
Note that for $k+1 \leq i \leq d, k \leq j\leq d$, by the induction hypothesis,
\begin{align*}
f^{(k)}_i(x)-\frac{k}{k+1}f^{(k)}_{i-1}(x)=&\binom{x}{k} x^{i-k} (x+1)^{d-i} - \frac{k}{k+1} \binom{x}{k} x^{i-k-1}(x+1)^{d-i+1}\\
=& x^{i-k-1}(x+1)^{d-i} \binom{x}{k} \left(x-\frac{k}{k+1}(x+1)\right)\\
=& x^{i-k-1} (x+1)^{d-i} \binom{x}{k} \frac{x-k}{k+1}
=\binom{x}{k+1} x^{i-k-1} (x+1)^{d-i}.
\end{align*}
Thus we derive that 
$$
a_d^{(k+1)}(i,j)=\left(T_d\left[\binom{x}{k+1} x^{i-k-1} (x+1)^{d-i}\right]\right)_j,
$$
and the proof is completed.
\end{proof}

Combining~\eqref{nemulti} and~\eqref{multi},
the multipliers in the NE are 
$$
m_{i,j}=\frac{a_d^{(j)}(i,j)}{a_d^{(j)}(i-1,j)}=\frac{j}{j+1}.
$$
Thus by Lemma \ref{NEofAd}, we obtain a matrix factorization 
$$
A_d=F_{d-1} F_{d-2} \cdots F_{0} U_d,
$$
where
\begin{equation*}
F_i =
\begin{tikzpicture}[baseline=(M.center)]
\matrix (M) [
matrix of math nodes,
nodes in empty cells,
left delimiter={[},
right delimiter={]},
column sep=0.7em,
row sep=0.35em,
nodes={
anchor=center,
inner sep=0pt,
minimum width=1.1em,
minimum height=1.05em,
font=\scriptsize
}
] {
1 &        &             &             &          & &      \\
0 & \ddots &             &             &          & &      \\
& \ddots & 1           &             &          & &      \\
&        & \frac{0}{1} & 1           &          & &      \\
&        &             & \frac{1}{2} & \ddots   & &      \\
&        &             &             & \ddots   & 1 &    \\
&        &             &             &          & \frac{d-i-1}{d-i} & 1\\
};

\node[font=\scriptsize] at ($(M-1-3.north)+(0,10pt)$) {$i$};
\node[font=\scriptsize, anchor=west] at ($(M-4-7.east)+(10pt,0)$) {$i+1$};
\end{tikzpicture},
\end{equation*}
and 
$$
U_d=\left[T_d(f^{(i)}_i(x))_j\right]_{0\le i\le j\le d}
$$
is an upper triangular matrix.
By~\eqref{eNEforAd},
the entry of $U_d=[u_d(i,j)]_{0\le i,j\le d}$ is
\begin{equation}\label{expud}
u_d(i,j)=T_d(f^{(i)}_i(x))_j= \sum_{r=0}^{j} \binom{r}{i} (r+1)^{d-i} (-1)^{j-r} \binom{d+1}{j-r}.
\end{equation}
See Example~\ref{deco:A4} for the case $d=4$.

%---------------------------------------------------------------------
\subsection{Bidiagonal decomposition of $U_d$}
%---------------------------------------------------------------------
It suffices to give a bidiagonal decomposition of $U_d$.
First we claim that $u_d(i,j)$ satisfies the following recurrence relation
\begin{equation}\label{recu_u}
u_d(i,j) = (j+1)\, u_{d-1}(i,j) + (d-j) \,u_{d-1}(i,j-1),
\end{equation}
with $0 \leq i \leq j \leq d-1$. 
Note that
\begin{align*}
(r+1)\binom{d+1}{j-r}&=
(r+1)\binom{d}{j-r}+(r+1)\binom{d}{j-r-1} \\
&=
(r+1)\binom{d}{j-r}
+
\bigl(d-(j-r)+1-d+j\bigr)\binom{d}{j-r-1}
\end{align*}
Using
$$
k\binom{d}{k}
=
(d-k+1)\binom{d}{k-1},
$$
we obtain that
$$
(r+1)\binom{d+1}{j-r}
=
(j+1)\binom{d}{j-r}
-
(d-j)\binom{d}{j-r-1}.
$$
Hence, by~\eqref{expud},
\begin{align*}
u_d(i,j)
&=
\sum_{r=0}^{j}
\binom{r}{i}
(r+1)^{d-1-i}
(-1)^{j-r}
(r+1)\binom{d+1}{j-r} \\
&=
\sum_{r=0}^{j}
\binom{r}{i}
(r+1)^{d-1-i}
(-1)^{j-r}
\left[
(j+1)\binom{d}{j-r}
-
(d-j)\binom{d}{j-r-1}
\right] \\
&=
(j+1)
\sum_{r=0}^{j}
\binom{r}{i}
(r+1)^{d-1-i}
(-1)^{j-r}
\binom{d}{j-r} \\
&\qquad
-
(d-j)
\sum_{r=0}^{j}
\binom{r}{i}
(r+1)^{d-1-i}
(-1)^{j-r}
\binom{d}{j-r-1}.
\end{align*}	
For the second sum, since \(\binom{d}{j-r-1}=0\) when \(r=j\), we may write
$$
\begin{aligned}
&-
(d-j)
\sum_{r=0}^{j}
\binom{r}{i}
(r+1)^{d-1-i}
(-1)^{j-r}
\binom{d}{j-r-1} \\
&\qquad =
(d-j)
\sum_{r=0}^{j-1}
\binom{r}{i}
(r+1)^{d-1-i}
(-1)^{j-1-r}
\binom{d}{j-1-r}.
\end{aligned}
$$
Therefore,
$$
\begin{aligned}
u_d(i,j)
&=
(j+1)
\sum_{r=0}^{j}
\binom{r}{i}
(r+1)^{d-1-i}
(-1)^{j-r}
\binom{d}{j-r} \\
&\qquad
+
(d-j)
\sum_{r=0}^{j-1}
\binom{r}{i}
(r+1)^{d-1-i}
(-1)^{j-1-r}
\binom{d}{j-1-r} \\
&=
(j+1)u_{d-1}(i,j)
+
(d-j)u_{d-1}(i,j-1).
\end{aligned}
$$
This proves the desired recurrence.

Let $B_d$ be the $d \times d$ upper bidiagonal matrix indexed by $0, \dots, d-1$, defined by
\begin{equation*}
B_d(i,j)= 
\begin{cases} 
j+1, \,& j = i, \\ 
d-j, \,& j = i+1, \\ 
0, \,& \text{otherwise}. 
\end{cases}
\end{equation*}
For $0 \leq i \leq j \leq d-1$, by the recurrence relation \eqref{recu_u}, we have
\begin{align*}
u_{d}(i,j)&=(j+1) \cdot u_{d-1}(i,j) + (d-j) \cdot u_{d-1}(i,j-1)\\
&=u_{d-1}(i,j) \cdot B_d(j,j) + u_{d-1}(i,j-1)\cdot B_d(j-1,j).
\end{align*}
By~\eqref{expud}, it is obvious that 
$$
u_d(d,d)=\sum_{r=0}^{d} \binom{r}{d} (r+1)^{0} (-1)^{d-r} \binom{d+1}{d-r}=1.
$$
This gives the following decomposition
\begin{equation}\label{recu_Ud}
U_d = \begin{bmatrix} U_{d-1} B_d & 0 \\ 0 & 1 \end{bmatrix}.
\end{equation}
Iterating the decomposition, we obtain
\begin{equation}\label{deco:Ud}
\begin{aligned}
&U_d= \begin{bmatrix} U_{d-1} B_d & 0 \\ 0 & 1 \end{bmatrix}
=\begin{bmatrix} \begin{bmatrix} U_{d-2} B_{d-1} & 0 \\ 0 & 1 \end{bmatrix} B_d & 0 \\ 0 & 1 \end{bmatrix}
=\begin{bmatrix} U_{d-2} B_{d-1} & 0 & 0 \\ 
0 & 1 & 0\\
0 & 0 & 1 \end{bmatrix}
\begin{bmatrix} B_d & 0 \\ 0 & 1 \end{bmatrix}\\
&=\begin{bmatrix}\begin{bmatrix} U_{d-3} B_{d-2} & 0 \\ 0 & 1 \end{bmatrix} B_{d-1} & 0 & 0 \\ 
0 & 1 & 0\\
0 & 0 & 1 \end{bmatrix}
\begin{bmatrix} B_d & 0 \\ 0 & 1 \end{bmatrix}
=\begin{bmatrix} U_{d-3} B_{d-2} & 0 & 0 & 0\\ 
0 & 1 & 0&0\\
0 & 0 & 1&0\\
0&0&0&1 \end{bmatrix} 
\begin{bmatrix} B_{d-1} & 0 & 0 \\ 
	0 & 1 & 0\\
	0 & 0 & 1 \end{bmatrix}
	\begin{bmatrix} B_d & 0 \\ 0 & 1 \end{bmatrix}\\
&=\dots=\begin{bmatrix} B_{1} & &  &\\ 
& 1 & &\\
&  &\ddots &\\
&&&1 \end{bmatrix} \cdots
\begin{bmatrix} B_{d-1} & 0 & 0 \\ 
0 & 1 & 0\\
0 & 0 & 1 \end{bmatrix}\begin{bmatrix} B_d & 0 \\ 0 & 1 \end{bmatrix}= \prod_{k=1}^{d} \widetilde{B}_k,
\end{aligned}
\end{equation}
where $\widetilde{B}_k$ is the direct sum of $B_k$ and the identity matrix of order $d+1-k$. 
Therefore, we give the bidiagonal decomposition of $U_d$.

\begin{exm}\label{deco:ud}
Taking $d=4$ as an example. 
We have shown in Example~\ref{deco:A4} that 
$$
U_4 =\mat{
1&11&11&1&0\\
0&8&14&2&0\\
0&0&9&3&0\\
0&0&0&4&0\\
0&0&0&0&1
}.
$$
The bidiagonal decomposition of $U_4$ is 
\begin{align*}
U_4=&\widetilde{B}_1 \widetilde{B}_2 \widetilde{B}_3 \widetilde{B}_4 \\
=& 
\left[
\begin{array}{c:cccc}
1& & & & \\
\hdashline
&1& & & \\
& &1& & \\
& & &1& \\
& & & &1
\end{array}
\right]
\left[
\begin{array}{cc:ccc}
1&1& & & \\
&2& & & \\
\hdashline
& &1& & \\
& & &1& \\
& & & &1
\end{array}
\right]
\left[
\begin{array}{ccc:cc}
1&2& & & \\
&2&1& & \\
& &3& & \\
\hdashline
& & &1& \\
& & & &1
\end{array}
\right]
\left[
\begin{array}{cccc:c}
1&3& & & \\
&2&2& & \\
& &3&1& \\
& & &4& \\
\hdashline
& & & &1
\end{array}
\right].
\end{align*}
\end{exm}

%---------------------------------------------------------------------
\subsection{Proof of Theorem~\ref{key}}
%---------------------------------------------------------------------
We obtain the bidiagonal decomposition of $A_d$ as
$$
A_d=F_{d-1} F_{d-2} \cdots F_{0} \widetilde{B}_1 \widetilde{B}_2\cdots \widetilde{B}_d.
$$
Recall that $H_\mathcal{F}$ is the transpose of $A_d$. 
Then we derive the bidiagonal decomposition of $H_\mathcal{F}$,
$$
H_\mathcal{F}=\widetilde{B}_d^T \cdots \widetilde{B}_2^T \widetilde{B}_1^T F_0^T \cdots F_{d-2}^T F_{d-1}^T.
$$
Taking $W_i=\widetilde{B}_{d-i}^T$ for $i=0,1,\ldots, d-1$,
and $W_i= F_{i-d}$ for $i=d,d+1,\ldots 2d-1$,
we complete the proof. 

Following Example~\ref{deco:A4} and Example~\ref{deco:ud},
we present the bidiagonal factorization of $H_\mathcal{F}$ for $d=4$,
\begin{align*}
H_\mathcal{F} = &\mat{
1 &   &   &   &   \\
3 & 2 &   &   &   \\
& 2 & 3 &   &   \\
&   & 1 & 4 &   \\
&   &   &   & 1
}
\mat{
1 &   &   &   &   \\
2 & 2 &   &   &   \\
& 1 & 3 &   &   \\
&   &   & 1 &   \\
&   &   &   & 1
}
\mat{
1 &   &   &   &   \\
1 & 2 &   &   &   \\
&   & 1 &   &   \\
&   &   & 1 &   \\
&   &   &   & 1
}
\mat{
1 &   &   &   &   \\
& 1 &   &   &   \\
&   & 1 &   &   \\
&   &   & 1 &   \\
&   &   &   & 1
}\cdot\\
& \mat{
1 & 0   &        &        &        \\
& 1 & \frac{1}{2} &        &        \\
&   & 1      & \frac{2}{3} &        \\
&   &        & 1      & \frac{3}{4} \\
&   &        &        & 1
}
\mat{
1 &   &   &        &        \\
& 1 & 0  &        &        \\
&   & 1 & \frac{1}{2} &        \\
&   &   & 1      & \frac{2}{3} \\
&   &   &        & 1
}
\mat{
1 &   &   &   &        \\
& 1 &   &   &        \\
&   & 1 & 0  &        \\
&   &   & 1 & \frac{1}{2} \\
&   &   &   & 1
}
\mat{
1 &   &   &   &   \\
& 1 &   &   &   \\
&   & 1 &   &   \\
&   &   & 1 &  0  \\
&   &   &   & 1
}.
\end{align*}

%---------------------------------------------------------------------
\section{Total positivity of the transformation matrix of the interval subdivision}
\label{sec4}
%---------------------------------------------------------------------
In this section, we study the total positivity of the transformation matrix of the interval subdivision. 
The interval subdivision of a simplicial complex was introduced by Walker~\cite{Wal88} in his work on canonical homeomorphisms of posets.
It is closely related to the barycentric subdivision~\cite{AN20}.
Let 
\begin{equation*}
H_{\mathcal{F}} = \left [p_d(i,j)\right ]_{0 \le i,j \le d}\quad \textrm{and}\quad H_{\mathcal{F}^*} = \left [q_d(i,j)\right ]_{0 \le i,j \le d}
\end{equation*}
be the transformation matrices of the barycentric subdivision and interval subdivision of $(d-1)$-dimensional simplicial complexes, respectively. 
Let 
$$
P_{d,j}(x) \coloneqq \sum_{i=0}^{d}p_d(i,j)x^i\quad \textrm{and} \quad  Q_{d,j}(x) \coloneqq \sum_{i=0}^{d}q_d(i,j)x^i
$$ 
be the $j$-th column generating functions of $H_{\mathcal{F}}$ and $H_{\mathcal{F}^*}$, respectively. 
The following lemma gives the relationship between $P_{d,j}(x)$ and $Q_{d,j}(x)$.

\begin{lem}[{\cite[Corollary 4.7]{AN20}}]\label{deco_polyB}
For $d \geq 0$ and $0 \leq j \leq d$, we have
\begin{equation}\label{A-B}
Q_{d, j}(x) = E_2\left[(1+x)^d P_{d, j}(x)\right],
\end{equation}
where $E_r$ is the operator on formal series defined by 
$$
E_2\left[\sum_{k \ge 0} c_k x^k\right] = \sum_{k \ge 0} c_{2k} x^k.
$$ 
\end{lem}

Now we present our result.
\begin{thm}\label{TP_int}
Let $\mathcal{F}^*$ be the interval subdivision of $(d-1)$-dimensional simplicial complexes. 
Then the transformation matrix $H_{\mathcal{F}^*}$ is totally positive.
\end{thm}

\begin{proof}
Taking the coefficients of $x^i$ on both sides of~\eqref{A-B},
we obtain that
\begin{equation}\label{A-B-C}
q_d(i, j)=\sum_{r=0}^{d} \binom{d}{2i-r} p_d(r,j).
\end{equation}
Let $S_d=\left[s_d(i,j)\right]_{\substack{0\le i\le d\\ 0\le j \le 2d}}$ be a row-selection matrix with
\begin{equation*}
s_d(i,j)=
\begin{cases}
1, & j=2i,\\
0, & j \neq 2i,
\end{cases}
\end{equation*}
and 
$$
V_d=\left[\binom{d}{i-j}\right]_{\substack{0\le i\le 2d\\ 0\le j \le d}}.
$$
By~\eqref{A-B-C}, we obtain the matrix factorization 
\begin{equation*}
H_{\mathcal{F}^*}=S_d V_d H_{\mathcal{F}}.
\end{equation*}

The TP of $S_d$ is obvious since all its minors are either $0$ or $1$.
Note that $V_d$ is a truncated submatrix of the Toeplitz matrix of the sequence $\left \{ \binom{d}{k} \right \}_{k=0}^{d}$. 
Since the generating function has only real non-positive zeros $\left(\sum_k \binom{d}{k}x^k=(1+x)^d\right)$, 
this sequence is a finite P\'olya frequency sequence.
By the Aissen–Schoenberg–Whitney theorem \cite{ASW52}, the Toeplitz matrix of 
this sequence is totally positive and so is the truncated submatrix $V_d$.
By Corollary~\ref{TP-sub}, the transformation matrix $H_{\mathcal{F}}$ is TP.
Then, by the Cauchy--Binet formula, 
the transformation matrix $H_{\mathcal{F}^*}$ of the interval subdivision is totally positive.
\end{proof}

The following example presents the decomposition of $H_{\mathcal{F^*}}$ for $d=4$.
\begin{exm}
Based on Example \ref{deco:A4} and the proof of Theorem~\ref{TP_int}, we can obtain that
\begin{equation*}
S_4=
\mat{
1&0&0&0&0&0&0&0&0\\
0&0&1&0&0&0&0&0&0\\
0&0&0&0&1&0&0&0&0\\
0&0&0&0&0&0&1&0&0\\
0&0&0&0&0&0&0&0&1
}, \,
V_4=
\mat{
1&0&0&0&0\\
4&1&0&0&0\\
6&4&1&0&0\\
4&6&4&1&0\\
1&4&6&4&1\\
0&1&4&6&4\\
0&0&1&4&6\\
0&0&0&1&4\\
0&0&0&0&1
}, \,
H_{\mathcal{F}}=
\mat{
1  & 0  & 0  & 0  & 0 \\
11 & 8  & 4  & 2  & 1 \\
11 & 14 & 16 & 14 & 11 \\
1  & 2  & 4  & 8  & 11 \\
0  & 0  & 0  & 0  & 1
}.
\end{equation*}
Then for $d=4$, we have
\begin{equation*}
H_{\mathcal{F}^*}=S_4 V_4 H_{\mathcal{F}}=
\mat{
1&0&0&0&0\\
61&46&32&22&15\\
115&124&128&124&115\\
15&22&32&46&61\\
0&0&0&0&1
}.
\end{equation*}
\end{exm}
%---------------------------------------------------------------------
\section{A criterion for TP$_2$ of the transformation matrix}
\label{sec5}
%---------------------------------------------------------------------
In this section, we show that if a uniform subdivision has the interlacing property, then the transformation matrix $H_\mathcal{F}$ is TP$_2$. 
This answers part of the question of Mu and Welker~\cite[Question 2.9]{MW26}. 
Before presenting the main result, we need some notions. 

Let $f,g\in \mathbb{R}[x]$ be two real-rooted polynomials with positive leading coefficients. 
Following Br{\"a}nd{\'e}n~\cite{Bra15}, we say that $f$ is an \emph{interleaver} of $g$ (written $f \preceq g$) if
$$
\cdots \le r_2(f) \le r_2(g)\le r_1(f)\le r_1(g),
$$
where $r_i(f)$ and $r_i(g)$ are the zeros of $f$ and $g$ in non-increasing order, respectively. 
For notational convenience, let $0 \preceq 0$, $0 \preceq h$ and $h \preceq 0$, where $h$ is any real-rooted polynomial with positive leading coefficient. 
By definition, if $f\preceq g$, then $\deg(g)=\deg(f)$ or $\deg(f)+1$. 
A sequence $(f_0(x), f_1(x), \dots , f_n(x))$ of real-rooted polynomials is called \emph{interlacing} if $f_i(x) \preceq f_j(x)$ for all $0 \le i < j \le n$. 
Note that if $(f_0(x), f_1(x), \dots , f_n(x))$ is interlacing,
then either $\deg(f_i(x))=\deg(f_j(x))$ for all $i,j$, or there exists an integer $m$ such that 
$$
\deg(f_1(x))+1=\cdots=\deg(f_m(x))+1=\deg(f_{m+1}(x))=\cdots=\deg(f_{n}(x)).
$$

Let $\mathcal{F}$ be a uniform subdivision of $(d-1)$-dimensional simplicial complexes, and 
$$
H_\mathcal{F}=\left [ p_{d}(i,j) \right ]_{0 \le i,j \le d}
$$ 
be its transformation matrix. 
Let 
$$
P_{d,k}(x)=\sum_{i=0}^{d}p_{d}(i,k)x^i
$$ 
be the $k$-th column generating function of the matrix $H_{\mathcal{F}}$. 
We say that $\mathcal{F}$ has the \emph{interlacing property} if the polynomial sequence $\left(P_{d,0}(x),P_{d,1}(x), \dots, P_{d,d}(x)\right)$ is an interlacing sequence. 
The following theorem is the main result of this section.

\begin{thm}\label{TP2ques}
If the uniform subdivision $\mathcal{F}$ has the interlacing property, then 
the transformation matrix $H_\mathcal{F}$ is TP$_2$.
\end{thm}
We prove this theorem by relating the interlacing property to the TP$_2$ of coefficient matrices. 
Given a polynomial sequence $(f_0(x), f_1(x), \dots,f_n(x))$, 
where 
$$
f_i(x)=\sum_{j=0}^{n}a({i,j})x^j,
$$
its \emph{coefficient matrix} is defined as $A=\left[a({i,j})\right]_{0\le i,j\le n}$. 
In other words, the polynomial $f_i(x)$ is the row generating function of $A$. 
Fisk gave the following result on the TP$_2$ of coefficient matrices.

\begin{lem}[{\cite[Lemma 3.20]{Fisk}}]
\label{fisklem}
Let $(f_0(x), f_1(x), \dots,f_n(x))$ be an interlacing polynomial sequence. 
Suppose that $\deg(f_i(x))=n$ and $f_i(x)$ has positive coefficients for each $i$. 
Then its coefficient matrix is TP$_2$.
\end{lem}

We will establish a strengthened version of this result. 
Before that, we need the following two lemmas. 
The first is a classical result due to Fekete (see~\cite[Chapter 2]{Pin10} for instance).

\begin{lem}[{\cite[Fekete's criterion for TP$_2$]{Pin10}}]
\label{Fekete}
A matrix with non-negative entries is TP$_2$ if all 
$2\times2$ minors with consecutive  rows and consecutive columns are non-negative.
\end{lem}

The second is based on the interlacing polynomials due to Fisk (see Lemma 1.20 and Remark 1.21 in \cite{Fisk}).

\begin{lem}
\label{Fiskpolylem}
Let $f(x)$ and $g(x)$ be polynomials with positive leading coefficients, where $g(x)$ has degree $n$ and roots $\{s_1, \dots, s_n\}$. 
If $f(x)$ has degree $n-1$ and we write
$$
f(x)=\sum_{r=1}^{n}c_r \frac{g(x)}{x+s_r},
$$
then $f(x) \preceq g(x)$ if and only if all $c_r\ge0$.
\end{lem}

Now we are ready to give a generalization of Lemma~\ref{fisklem}.

\begin{prop}\label{CondTP2}
Let $(f_0(x), f_1(x), \dots,f_n(x))$ be an  interlacing polynomial sequence. 
Suppose that $f_i(x)$ has non-negative coefficients for each $i$. 
Then its coefficient matrix is TP$_2$.
\end{prop}

\begin{proof}
By Lemma~\ref{Fekete},
to prove the coefficient matrix $A$ is TP$_2$, 
it suffices to prove that the submatrix of $A$ with the consecutive two rows
$$
A'=\mat{
a({i,0})&a({i,1})&\dots&a({i,n})\\
a({i+1,0})&a({i+1,1})&\dots&a({i+1,n})
}
$$
is TP$_2$.
We divide the proof into two steps.

Step 1: Assume that $a(i,0)>0$ and $a(i+1,0)>0$.
Without loss of generality, we assume that $a({i+1,n})>0$.
It has been shown that $\deg(f_i(x))=\deg(f_{i+1}(x))$ or $\deg(f_i(x))+1=\deg(f_{i+1}(x))$.
For the first case, it is obvious that $a({i,n})>0$ and the matrix is TP$_2$ by Lemma~\ref{fisklem}.
For the second case, we have $a({i,n})=0$ and $a({i,n-1})>0$.
Then the entries of $A'$ are strictly positive except $a({i,n})=0$.

Write
$$
f_{i+1}(x)=a({i+1,n})\,\prod_{r=1}^{n}(x+s_r),
\quad
0<s_1\le s_2\le \cdots\le s_n.
$$
Define $\widetilde{f}_{i+1}(x)=(x+t)f_{i+1}(x)$,
where $t>s_n$.
It is easy to check that
$$
\left( \frac{\widetilde{f}_{i+1}(x)}{x+s_1}, \frac{\widetilde{f}_{i+1}(x)}{x+s_2}, \dots \frac{\widetilde{f}_{i+1}(x)}{x+s_n}, \frac{\widetilde{f}_{i+1}(x)}{x+t}=f_{i+1}(x) \right)
$$
forms an interlacing polynomial sequence. 
All these polynomials have degree $n$ and their coefficients are positive. 
Therefore, by Lemma~\ref{fisklem}, their coefficient matrix $B_t$ is TP$_2$. 

For $r=1,2,\dots,n$, let
$$
h_r(x)=\frac{f_{i+1}(x)}{x+s_r}=\frac{\widetilde{f}_{i+1}(x)}{(x+t)(x+s_r)}.
$$
Then
$$
\frac{\widetilde{f}_{i+1}(x)}{(x+s_r)}=(x+t)h_r(x).
$$
Consider the coefficient matrix $C_t$ of
$$
\left(
\left(1+\frac{x}{t}\right)h_1(x),\ 
\left(1+\frac{x}{t}\right)h_2(x),\ 
\ldots,\ 
\left(1+\frac{x}{t}\right)h_n(x),\ 
f_{i+1}(x)
\right).
$$
Clearly, $C_t$ is obtained from $B_t$ by dividing the first $n$ rows by the positive scalar $t$. 
Since $B_t$ is TP$_2$, we have that $C_t$ is TP$_2$. 
As $t\to\infty$,
$$
\left(1+\frac{x}{t}\right)h_r(x)\longrightarrow h_r(x)
$$
coefficientwise for each $r$, while the last row of $C_t$ is independent of $t$. Therefore
$$
C_t\longrightarrow C
$$
entrywise. 
Since every entry and every $2\times2$ minor is a continuous function of the entries, the class of TP$_2$ matrices is closed under entrywise limits. 
Hence $C$ is TP$_2$. 
By Lemma~\ref{Fiskpolylem}, since $f_i(x) \preceq f_{i+1}(x)$, and $\deg(f_i(x))+1=\deg(f_{i+1}(x)),$ 
we have
$$
f_i(x)=\sum_{r=1}^{n}c_r \frac{f_{i+1}(x)}{x+s_r},
$$
where $c_r\ge0$. 
Thus we obtain that
$$
A'=
\begin{pmatrix}
c_1&c_2&\cdots&c_n& 0\\
0&0&\cdots&0&1
\end{pmatrix} C.
$$
By the Cauchy--Binet formula, the TP$_2$ of $A'$ follows that of $C$.

{Step 2: Suppose that $f_i(x)=x^{m_1} g_i(x)$ and $f_{i+1}(x)=x^{m_2} g_{i+1}(x)$,
where $g_i(0)\ne 0$ and $g_{i+1}(0)\ne 0$.} 
Then $m_1=m_2$ or $m_1+1=m_2$
by $f_i(x)\preceq f_{i+1}(x)$.
Without loss of generality, we only need to consider the case $a({i,0})> 0$. 
If $a(i+1,0)> 0$, then the matrix is TP$_2$.
If $a(i+1,0)=0$ and $a(i+1,1)>0$, then
$$
f_{i+1}(x)=a({i+1,n})\,x\, \prod_{r=2}^{n}(x+s_r),
\quad
0=s_1<s_2\le s_3\le \cdots\le s_n.
$$
We only consider the case $\deg(f_i(x))=n$ and the case $\deg(f_i(x))=n-1$ is similar. 
Write
$$
f_{i}(x)=a({i,n})\, \prod_{r=1}^{n}(x+t_r),
\quad
0\le t_1\le t_2\le t_3\le \cdots\le t_n.
$$
Let $\varepsilon > 0$ be a sufficiently small perturbation parameter and define two polynomials
\begin{equation*}
g_i(x) = a({i,n})\, \prod_{r=1}^{n}(x+t_r+\varepsilon)
\end{equation*}
and 
$$
g_{i+1}(x)=a({i+1,n})\,(x+\varepsilon)\, \prod_{r=2}^{n}(x+s_r+\varepsilon).
$$
Then $g_i(x)\preceq g_{i+1}(x)$.
Because every factor in the product has strictly positive coefficients,
by Step 1, the coefficient matrix
$$
A^{\varepsilon}=\mat{
a^{\varepsilon}({i,0})&a^{\varepsilon}({i,1})&\cdots&a({i,n})\\
a^{\varepsilon}({i+1,0})&a^{\varepsilon}({i+1,1})&\cdots&a({i+1,n})
}
$$
of $g_i(x)$ and $g_{i+1}(x)$ is TP$_2$.
Taking the limit as $\varepsilon \to 0$, we have
\begin{equation*}
\lim_{\varepsilon \to 0} a^{\varepsilon}(i,j) = a({i,j}),
\end{equation*}
and
\begin{align*}
&\lim_{\varepsilon \to 0} \left( a^{\varepsilon}(i,j) a^{\varepsilon}(i+1,j+1) - a^{\varepsilon}(i,j+1) a^{\varepsilon}(i+1,j) \right) \\
=~& a(i,j) a(i+1,j+1) - a(i,j+1) a(i+1,j) \ge 0.
\end{align*}
Then, by Lemma~\ref{Fekete}, we obtain that $A'$ is TP$_2$. This completes the proof.
\end{proof}

\begin{proof}[Proof of Theorem \ref{TP2ques}]
By Theorem 1.1 of~\cite{Ath22}, the coefficient matrix of the polynomial sequence
$$
\left(P_{d,0}(x),P_{d,1}(x), \dots, P_{d,d}(x)\right)
$$
is the transpose of $H_{\mathcal{F}}$. 
Since $\mathcal{F}$ has the interlacing property, this sequence is interlacing. 
By Proposition~\ref{CondTP2}, the coefficient matrix of this sequence, namely $H_\mathcal{F}^{T}$, is TP$_2$. 
As total positivity is preserved under transposition (see Proposition~1.2 in \cite{Pin10}), it follows that $H_\mathcal{F}$ is TP$_2$. 
This completes the proof.
\end{proof}

The $r$-colored barycentric subdivision $\mathcal{F}$ was introduced in~\cite{Ath14}.
The barycentric subdivision and the interval subdivision are recovered as the special cases $r=1$ and $r=2$, respectively. 
Athanasiadis and Tzanaki investigated the interlacing property of $H_\mathcal{F}$ for the $r$-colored barycentric subdivision.

\begin{lem}[{\cite[Theorem 5.3]{AT21}}]
Let $\mathcal{F}$ be the $r$-colored barycentric subdivision of $(d-1)$-dimensional simplicial complexes.
Then for each $d$ and $r$, $\mathcal{F}$ has the interlacing property.
\end{lem}

The following follows immediately from Theorem~\ref{TP2ques}.
\begin{coro}
Let $\mathcal{F}$ be the $r$-colored barycentric subdivision of $(d-1)$-dimensional simplicial complexes.
Then for each $d$ and $r$,
the transformation matrix $H_\mathcal{F}$ is TP$_2$.
\end{coro}

\begin{rem}
Another typical uniform subdivision is the antiprism triangulation, introduced by Izmestiev and Joswig~\cite{IJ03} to study combinatorially branched coverings of manifolds. 
Recently, Athanasiadis, Brunink and Juhnke-Kubitzke~\cite{ABJ22} studied the enumerative and algebraic properties of antiprism triangulations. 
In particular, from the results in Section~5 of~\cite{ABJ22}, we can obtain an explicit expression for the entries of the transformation matrix $H_\mathcal{F}$ as
$$
H_\mathcal{F}=
\left[ \sum_{r=0}^{d}\sum_{s=0}^{i} \binom{j}{s}\binom{d-j}{d-r-s}\binom{r}{i-s} \sum_{\ell\geq 0}(\ell!)^2 S(i-s,\ell)S(r-i+s,\ell) \right]_{0\leq i,j\leq d},
$$
where $S(a,b)$ denotes the Stirling number of the second kind.
It is natural to consider the interlacing property and total positivity of this
transformation matrix.
\end{rem}
%---------------------------------------------------------------------

\section*{Declaration of competing interest}
The authors declare no conflict of interest. 

\section*{Declaration of generative AI and AI-assisted technologies in the writing process}
The authors didn't use Generative AI and AI-assisted technologies in the writing process.

\section*{Acknowledgements}
This work was partially supported by the National Natural Science Foundation of China (No.12201100).
%-----------------------------------------


\begin{thebibliography}{00}

\bibitem{Ath14}
C.A. Athanasiadis, 
\textit{Edgewise subdivisions, local $h$-polynomials, and excedances in the wreath product $\mathbb{Z}_r \wr \mathfrak{S}_n$}, 
SIAM J. Discrete Math. 
28 (2014) 1479--1492.

\bibitem{AT21}
C.A. Athanasiadis, E. Tzanaki, 
\textit{Symmetric decompositions, triangulations and real-rootedness}, 
Mathematika. 
67 (2021) 840--859.

\bibitem{Ath22}
C.A. Athanasiadis, 
\textit{Face numbers of uniform triangulations of simplicial complexes}, 
Int. Math. Res. Not. IMRN. 
(2022) no.~20, 15756--15787.

\bibitem{ABJ22}
C.A. Athanasiadis,  J.-M. Brunink, M. Juhnke-Kubitzke,
\textit{Combinatorics of antiprism triangulations}, 
Discrete Comput. Geom.
68(1) (2022) 72--106.

\bibitem{AN20}
I. Anwar, S. Nazir, 
\textit{The $f$- and $h$-vectors of interval subdivisions}, 
J. Combin. Theory Ser. A. 
169 (2020) 105124, 22 pp.

\bibitem{AN21}
I. Anwar, S. Nazir, 
\textit{The $\gamma$- and local $\gamma$-vectors of interval subdivisions}, 
J. Algebraic Combin.
53(3) (2021) 677--699.

\bibitem{ASW52}
M.I. Aissen, I.J. Schoenberg, A. Whitney, 
\textit{On the generating functions of totally positive sequences. I}, 
J. Analyse Math. 
2 (1952) 93--103.

\bibitem{Bra15}
P. Br{\"a}nd{\'e}n,
\textit{Unimodality, log-concavity, real-rootedness and beyond},
Handbook of enumerative combinatorics. 
CRC Press, Boca Raton, FL, 2015, pp. 437--483.

\bibitem{Bre95}
F. Brenti, 
\textit{Combinatorics and total positivity}, 
J. Combin. Theory Ser. A. 
71 (1995) 175--218.

\bibitem{BW06}
F. Brenti, V. Welker,
\textit{$f$-Vectors of barycentric subdivisions},
Math. Z.
259 (2008) no.~4, 849--865.

\bibitem{BW09}
F. Brenti, V. Welker, 
\textit{The Veronese construction for formal power series and graded algebras}, 
Adv. in Appl. Math. 
42 (2009) 545--556.

\bibitem{DF09}
P. Diaconis, J. Fulman, 
\textit{Carries, shuffling, and an amazing matrix}, 
Amer. Math. Monthly. 
116 (2009) 788--803.

\bibitem{Fisk}
S. Fisk,
\textit{Polynomials, roots, and interlacing},
arXiv:math/0612833v2.

\bibitem{GP92}
M. Gasca, J.M. Pe{\~n}a,
\textit{Total positivity and Neville elimination},
Linear Algebra Appl.
165 (1992) 25-44.

\bibitem{GP96}
M. Gasca, J.M. Pe{\~n}a,
\textit{On factorizations of totally positive matrices},
in: M. Gasca, C.A. Micchelli (Eds.),
Total Positivity and Its Applications,
Kluwer Academic Publishers,
Dordrecht, The Netherlands, 1996, pp. 109–130.

\bibitem{GT04}
M.T. Gass{\'o}, J.R. Torregrosa, 
\textit{A totally positive factorization of rectangular matrices by the Neville elimination}, 
SIAM J. Matrix Anal. Appl. 
25 (2004) 986--994.

\bibitem{IJ03}
I. Izmestiev, M. Joswig, 
\textit{Branched coverings, triangulations, and 3-manifolds}, 
Adv. Geom. 
3 (2003) 191--225.

\bibitem{JN09}
M. Juhnke-Kubitzke, E. Nevo, 
\textit{The Lefschetz property for barycentric subdivisions of shellable complexes}, 
Trans. Amer. Math. Soc. 
361 (2009) 6151--6163.

\bibitem{Kar68}
S. Karlin,
\textit{Total Positivity, Volume 1},
Stanford University Press, 1968.

\bibitem{KMRR24}
Y. Khiar, E. Mainar, E. Royo-Amondarain and B. Rubio, 
\textit{Total positivity and accurate computations related to $q$-Abel polynomials}, 
J. Sci. Comput. 
101 (2024) no.~3, Paper No. 56, 22 pp.

\bibitem{MPR22}
E. Mainar, J.M. Pe{\~n}a, B. Rubio, 
\textit{Accurate computations with Wronskian matrices of Bessel and Laguerre polynomials}, 
Linear Algebra Appl. 
647 (2022) 31--46.

\bibitem{MW22}
J. Mao, Y. Wang, 
\textit{Proof of a conjecture on the total positivity of amazing matrices}, 
Adv. Appl. Math. 
140 (2022) 102387.

\bibitem{MWZ26}
L. Mu, C. Wang, B.-X. Zhu,
\textit{On the totally positive matrices related to refined Eulerian numbers},
Discrete Math. 
349 (2026), no.~11, Paper No. 115247.

\bibitem{MW26}
L. Mu, V. Welker, 
\textit{Total positivity and two inequalities by Athanasiadis and Tzanaki}, 
Ann. Comb. 
30 (2026) no.~1, 59--75.

\bibitem{NPT11}
E. Nevo, T.K. Petersen, B.E. Tenner, 
\textit{The $\gamma$-vector of a barycentric subdivision}, 
J. Combin. Theory Ser. A. 
118 (2011) 1364--1380.

\bibitem{Pin10}
A. Pinkus, 
\textit{Totally positive matrices}, 
Cambridge Tracts in Mathematics, 181, 
Cambridge Univ. Press, Cambridge, 2010.

\bibitem{Sta92}
R.P. Stanley,
\textit{Subdivisions and local $h$-vectors}, 
J. Amer. Math. Soc. 5 (1992) 805--851.

\bibitem{Wal88}
J.W. Walker, 
\textit{Canonical homeomorphisms of posets}, 
European J. Combin. 
9 (1988) no.~2, 97--107.

\end{thebibliography}
\end{document}